\documentclass[12pt]{amsart}
\usepackage[colorlinks,linkcolor=blue,citecolor=blue,urlcolor=red]{hyperref}
\usepackage{amssymb,latexsym,amscd,verbatim,color}
\usepackage[all]{xy}

\swapnumbers
\newtheorem{thm}{Theorem}[subsection]
\newtheorem{propose}[thm]{Proposition}
\newtheorem{lemma}[thm]{Lemma}
\newtheorem{cor}[thm]{Corollary}

\theoremstyle{definition}
\newtheorem{defn}[thm]{Definition}

\newtheorem{remark}[thm]{Remark}

\numberwithin{equation}{section}
\newcounter{spec}
{\end{list}}

\renewcommand{\d}{{\text{\LARGE $\cdot $}}}
\renewcommand{\hat}{\widehat}
\newcommand{\Spec}{\operatorname{Spec}} 
\newcommand{\DM}{\operatorname{DM}}          
          
\newcommand{\Shv}{\operatorname{Shv}}
\newcommand{\Ch}{\operatorname{Ch}}

\newcommand{\Sch}{\operatorname{Sch}}

\newcommand{\T}{\mathbb{T}} 
\newcommand{\Tmod}{\text{$\T$-{\sf Mod}}}
\newcommand{\TTmod}{\text{$\bar{\T}$-{\sf Mod}}}
\newcommand{\Ttmod}{\text{$\T_T$-{\sf Mod}}}
\newcommand{\Rmod}{\text{$R$-{\sf Mod}}}

\newcommand{\Top}{\text{$\T^{op}$-{\sf Mod}}}

\newcommand{\Gmod}{\text{$G$-{Mod}}}

\newcommand{\Sub}{\operatorname{Sub}}
\newcommand{\Ind}{\operatorname{Ind}}
\newcommand{\Ab}{\operatorname{Ab}}

\newcommand{\reg} {{\rm reg}}

\newcommand{\Aff}{\mathbb{A}}   
\newcommand{\A}{{\sf Add}}      

\newcommand{\C}{\mathbb{C}}     
     
\newcommand{\Q}{\mathbb{Q}}     
\newcommand{\Z}{\mathbb{Z}}     

\newcommand{\im}{\operatorname{Im}}        
\renewcommand{\ker}{\operatorname{Ker}}  

\newcommand{\Cor}{\operatorname{Cor}}

\newcommand{\Tot}{\operatorname{Tot}}     

\newcommand{\by}[1]{\stackrel{#1}{\rightarrow}}
\newcommand{\longby}[1]{\stackrel{#1}{\longrightarrow}}

\newcommand{\longyb}[1]{\stackrel{#1}{\longleftarrow}}

\renewcommand{\tilde}{\widetilde}
\newcommand{\df}{\mbox{\,${:=}$}\,}
\newcommand{\ie}{{\it i.e. }}
\newcommand{\cf}{{\it cf. }}
\newcommand{\eg}{{\it e.g. }}

\newcommand{\et} {{\operatorname{\acute{e}t}}}

\newcommand{\eff}{{\operatorname{eff}}}

\renewcommand{\bar}{\overline}
\newcommand{\into}{\hookrightarrow}

\newcommand{\limdir}[1]{\mathop{\rm
``lim"}_{\buildrel\longrightarrow\over{#1}}}

\renewcommand{\lim}{\varprojlim}

\newcommand{\onto}{\mbox{$\to\!\!\!\!\to$}}

\newcommand{\boxtensor}{\def\boxtimesten{\Box\kern-7.59pt\raise1.2pt
\hbox{$\times$} }}                                  

\newcounter{elno}                   

\newcommand{\cA}{\mathcal{A}}
\newcommand{\cB}{\mathcal{B}}
\newcommand{\cC}{\mathcal{C}}
\newcommand{\cD}{\mathcal{D}}
\newcommand{\cE}{\mathcal{E}}
\newcommand{\cF}{\mathcal{F}}

\newcommand{\cH}{\mathcal{H}}

\newcommand{\cM}{\mathcal{M}}

\newcommand{\cS}{\mathcal{S}}

\renewcommand{\phi}{\varphi}
\renewcommand{\epsilon}{\varepsilon}

\begin{document}
\title{$\T$-Motives}
\author{Luca Barbieri-Viale}
\address{Dipartimento di Matematica ``F. Enriques", Universit{\`a} degli Studi di Milano\\ Via C. Saldini, 50\\ I-20133 Milano\\ Italy}
\email{luca.barbieri-viale@unimi.it}

\begin{abstract}
Considering a (co)homology theory $\T$ on a  base category $\cC$ as a fragment of a first-order logical theory we here construct an abelian category $\cA[\T]$ which is universal with respect to models of $\T$ in abelian categories. Under mild conditions on the base category $\cC$, \eg for the category of algebraic schemes, we get a functor from $\cC$ to $\Ch(\Ind(\cA[\T]))$ the category of chain complexes of ind-objects of $\cA[\T]$. This functor lifts Nori's motivic functor for algebraic schemes defined over a subfield of the complex numbers. Furthermore, we construct a triangulated functor from $D(\Ind(\cA[\T]))$ to Voevodsky's motivic complexes.
\end{abstract}
\maketitle

\section*{Introduction}
The first task of this paper is to set out the framework of ``theoretical motives'' or $\T$-motives jointly with that of a ``motivic topos'' and to perform some general constructions. The second task is to make use of this framework in algebraic geometry relating $\T$-motives with Nori motives and Voevodsky motives. Let me immediately warn the reader about a possible misunderstanding: these ``theoretical motives'' won't be ``mixed motives'' at once. The concerned wishes are {\it i)}\, to present ``mixed motives'' as a Serre quotient of ``theoretical motives'' and/or {\it ii)}\, to present ``mixed motivic complexes'' as a Bousfield localization of ``theoretical motivic complexes''. In fact, adding $I^+$-invariance and $cd$-exactness to the (co)homology theory $\T$ should most likely be enough to get ``mixed motives''. The matters treated in this paper can be explained as follows.
 
\subsubsection*{(Co)homology theories} A (co)homology theory is herein considered as a fragment of a first-order \emph{logical theory} (see Section \ref{preth} for essential preliminaries in categorical logic). Models shall be families of internal (abelian) groups in a topos or in a suitable (\eg regular, Barr exact or abelian) category, satisfying some axioms. For a fixed base category $\cC$ along with a distinguished subcategory $\cM$ we mean that a model $H$ of such a homological theory in $\cE$ shall be at least a functor $$H: \cC^{\square}\to \cE \hspace*{0.5cm} (X,Y)\leadsto \{H_n (X,Y)\}_{n\in\Z}$$
where $(X,Y)$ is a notational abuse for $Y\to X\in \cM$ and $\cC^{\square}$ is the category of arrows whose objects are arrows of $\cM$. Cohomological functors shall be contravariant and homological functors shall be covariant. We actually here introduce a \emph{regular} (co)homology theory $\T$ (see Section~\ref{rehomth} for details) with a simple exactness axiom so that for any pair of composable arrows $Z\to Y\to X$ in $\cM$ any model $H$ shall be provided with a long exact sequence 
$$\cdots\to H_n (Y,Z)\to H_n (X,Z)\to H_n (X,Y)\to H_{n-1} (Y,Z)\to \cdots$$
of (abelian) groups in $\cE$. This key (co)homology theory $\T$ includes as models all known (co)homology theories (see Section~\ref{models} for a list of examples). Usual (co)homology theories on the category $\cC =\Sch_k$ of algebraic schemes, \eg singular (co)homology, can be recovered as models of the theory $\T$ in the category $\cS$ of sets. Suslin-Voevodsky singular homology can be recovered as a model of $\T$ in the category of (additive) presheaves with transfers. By the way we may obtain new theories $\T^\prime$ adding axioms to our theory $\T$. For example, there is always the regular theory $\T_H$ of a model $H$ obtained adding all regular axioms which are valid in the model $H$. Notably we indicate two additional regular axioms expressing the geometric nature of the homology theory: $I^+$-invariance and $cd$-exactness (see Definitions~\ref{homtopinv} and \ref{cdinv}).

\subsubsection*{Motivic topos} Treating (co)homology theories as first-order theories we can deal with a topos of sheaves on the corresponding syntactic site which we may call \emph{motivic topos}. For our regular theory $\T$ let $\cE[\T]$ be the topos of sheaves on the corresponding regular syntactic site  $\cC_{\T}^\reg$  (see \S \ref{syncat} and \S \ref{Scltop}). Denote $\gamma: \cE[\T]\to \cS$ the unique (geometric) morphism to $\cS$. Recall that $\cE [\T]$ is connected if $\gamma^*$ is fully faithful and locally connected if $\gamma^*$ has a left adjoint $\gamma_!$.  As $\T$ is a regular theory we get that $\cE [\T]$ is connected and locally connected (see Lemma~\ref{regcon}) so that we also get a \emph{$\T$-motivic Galois group} $G_\T$. Recall (\cf SGA~1, SGA~4 and precisely \cite[IV Ex. 2.7.5]{SGA4}) that for such a topos $\cE[\T]$ with a point $f$ then  $$G_\T\df\pi_1(\cE [\T] ,f)$$
is a pro-group such that $\cB_{G_\T}$ is equivalent to the full subcategory ${\sf Gal}(\cE [\T])$ of $\cE [\T]$ of locally constant objects. Note that the inclusion is the inverse image of a surjective morphism $\cE [\T]\onto {\sf Gal}(\cE [\T])$ and the topos $\cB_{G_\T}\cong {\sf Gal}(\cE [\T])$ is a Galois topos (see also \cite[App. A]{D} and \cite{JT}). For algebraic schemes, adding $I^+$-invariance and $cd$-exactness to $\T$, we still get a regular theory  and it will be interesting to compare the resulting motivic Galois group with the Ayoub-Nori motivic Galois groups (see \cite{GC}): we will treat this matter elsewhere.

However, we have that $\cE[\T]$ is a classifying topos, as for any regular theory $\T$ (see Theorem \ref{cltop}), \ie we have a natural equivalence 
$$\Tmod (\cE) \cong \text{\sf Hom} (\cE, \cE[\T])$$ 
between the category of $\T$-models in a topos $\cE$ and that of (geometric) morphisms from $\cE$ to $\cE[\T]$.  

\subsubsection*{$\T$-motives} Internally, in the motivic topos, we do have (abelian) groups that may be considered as ``motives'' in the following sense (see Section~\ref{motives}).  Given a model $H$ of our (co)homology theory $\T$ in a topos $\cE$ we get a morphism $f_H:\cE\to \cE[\T]$ and a \emph{realization} exact functor $$f^*_H : \Ab(\cE[\T])\to \Ab(\cE)$$ induced by the inverse image  $f^*_H$ of $f_H$. Note that the realization is also faithful if $f_H$ is a surjection. For theories $\T^\prime$ obtained adding axioms to $\T$ we get a sub-topos $f:\cE[\T^\prime]\into \cE[\T]$ and an induced  Serre quotient $f^*:\Ab(\cE[\T])\to \Ab(\cE[\T^\prime])$ with a section $f_*$ (see Lemma \ref{fadd}).

Recall that the category of abelian groups in a Grothendieck topos is a Grothendieck category. However, we may wish to restrict realizations to smaller abelian categories: this is possible in the case of the regular theory $\T$. Actually, the regular syntactic category $\cC_{\T}^\reg$ of our regular theory $\T$ is additive (and this is also the case for any other regular theory $\T^\prime$ on the same signature, see Lemma~\ref{synadd} and \cf \cite[Lemma 2.4]{BVCL}). Let $\cA [\T]$ be the (Barr) exact completion of $\cC_{\T}^\reg$: it is an abelian category.  Call (effective) \emph{constructible $\T$-motives} the objects of $\cA [\T]$ and \emph{$\T$-motives} the objects of $\Ind(\cA[\T])$. 

The use of the (Barr) exact completion of the regular syntactic category was an idea of O.\/ Caramello, appearing in \cite{BVCL} in order to obtain Nori's category of a representation of a diagram {\it via}\, the regular theory of a model (see Theorem \ref{BVCL}). In particular, this applies to the model given by singular homology and yields back Nori motives (see Corollary \ref{Nori}). However, the universal representation theorem of \cite{PBV} shows us that Nori's category as well as all these categories $\cA [\T^\prime]$ can be seen directly as Serre's quotients of Freyd's free abelian category on the preadditive category generated by a diagram. The link with the syntactic category is then given by the additive definable category generated by a model (see \cite{PBV} for details).  In general, here we have fully faithful exact functors $$\cA [\T]\into \Ind(\cA[\T])\into \Ab(\cE [\T])$$ 
and for any abelian category $\cA$ we get an equivalence
$$\Tmod (\cA) \cong \text{\sf  Ex} (\cA [\T], \cA)$$ 
with the category of exact functors (see Proposition \ref{abeliancon}). Given a $\T$-model $H$ in $\cA$ we thus get an exact functor $$r_H: \cA [\T]\to \cA$$
If $\cA =\Ind(\cA)$ then $r_H$ induces an exact functor $\Ind(\cA[\T])\to \cA$. Note that any abelian category $\cA$ can be regarded as an exact full subcategory of $\Ab(\cE)$ for a suitable topos $\cE$. Thus $r_H$ is always the restriction of the $f^*_H$ induced by the corresponding $\T$-model in $\cE$.
 
\subsubsection*{$\T$-motivic complexes} The category of $\T$-motivic complexes is the category $\Ch (\Ind (\cA[\T]))$ of (unbounded) chain complexes of $\T$-motives (see \S \ref{complexes}): this is  a target for a $\T$-motivic functor from our base category $\cC$. In order to get such a functor 
 $$\cC\to \Ch (\Ind (\cA[\T]))$$ we deal with a Grothendieck ``niveau'' spectral sequence (see Lemma \ref{niveau}) which measures the defect of cellularity (see Lemma \ref{cellular}). Note that for any $H\in \Tmod (\cA)$ in an abelian category $\cA$ the exact functor $r_H$ induces an exact functor $\Ind(\cA [\T]) \to \Ind(\cA)$ which yields a realization functor 
$$\Ch(\Ind (\cA[\T]))\to \Ch(\Ind (\cA))$$
(see Proposition \ref{eqrel}). The resulting $\T$-motivic functor actually lifts Nori's motivic functor (see Proposition \ref{Nmotf}). A similar construction yields a functor to Voevodsky motivic complexes (see Proposition \ref{DM}).

Finally, remark that we have a natural model category structure on the category of chain complexes over any Grothendieck abelian category where cofibrations are the
monomorphisms and weak equivalences the quasi-isomorphisms (see \cite{CD} and \cite{Bk}). We will study the relations with the universal model structure given by simplicial presheaves on $\cC$  in another paper.

\section*{Notation} We shall follow the conventions adopted in \cite{KS} and \cite{SGA4} on Grothendieck universes $U$. For example, we adopt the same definition of a $U$-category and that of a $U$-small category $\cC$ but we drop the reference to $U$ when unnecessary. For a topos we here mean a Grothendieck topos. Denote $\cS$ the topos of $U$-sets and $\Ab$ the category of $U$-abelian groups.  Denote $\hat{\cC}$ the {\it big} category of pre-sheaves of $U$-sets.  Denote $\Ind(\cC)$ the $U$-category of Ind objects of any $U$-category $\cC$ (\cf \cite[\S 6]{KS}).  For a cartesian category $\cC$ denote $\Ab(\cC)$ the category of internal abelian groups. Denote ${\sf Lex} (\cC, \cD)$ the category of left exact functors from $\cC$ to $\cD$. If $\cC$ and $\cD$ are additive categories denote $\A (\cC,\cD)$ the category of additive functors. For a site $(\cC, J)$ we always assume that $\cC$ is essentially $U$-small and has finite limits; we denote $\cE\df \Shv (\cC, J)$ the topos of sheaves of $U$-sets. Denote ${\sf Hom} (\cE, \cF)$ the category of geometric morphisms $f = (f_*, f^*)$ from a topos $\cE$ to a topos $\cF$. 

\section{Preliminaries on theories and models}\label{preth}
Recall briefly what is a general first-order theory $\T$ over a signature $\Sigma$ along with its categorical interpretations. We refer to \cite{El} for a modern detailed textbook but see \cite{BJ} for a synthetic exposition of the key facts. See also \cite{Bu} for a comprehensive account including detailed proofs of all key facts on {\it regular} theories and categories.

A signature $\Sigma$ consists of sorts $X,Y,\ldots$ function symbols $f,g,\ldots$ and relation symbols $R,S, \ldots$ (see\cite[D1.1]{El}). A collection of terms and formulas $\varphi$ is formed by allowing regular, geometric and general first-order formulas over $\Sigma$.  A theory $\T$ (see \cite[D1.1.3]{El}) is a set of sequents $\phi\vdash_{\vec{x}} \psi$ called axioms. Say that a theory $\T^\prime$ is an extension of a theory $\T$ if the theory $\T^\prime$ is obtained from $\T$ by adding axioms over a signature $\Sigma^\prime$ containing $\Sigma$. 

 A theory may be interpreted in a category (see \cite[D1.2]{El}). A $\T$-model is an interpretation such that all axioms are valid. We shall denote by  $\Tmod (\cC)$ the category of $\T$-models in a category $\cC$. Say that two theories $\T$ and $\T^\prime$ are Morita equivalent if they have equivalent categories of models. Varying categories of $\T$-models $\Tmod (-)$ can be made 2-functorial with respect to appropriate functors. 
  
A fragment of first-order logic that is particularly interesting for our purposes is that of a regular theory: this is a theory where all axioms are regular sequents, \ie involving formulas making use of $\top$, $=$, $\wedge$ and $\exists$ only (see \cite[D1.1.3(c)]{El} and \cite[\S 3]{Bu} for details).  A regular theory can be interpreted in any regular category.  For a regular theory $\T$ and a regular functor $\cC\to \cD$ between regular categories, \ie a left exact functor that preserves regular epis, we get a functor $\Tmod (\cC)\to \Tmod (\cD)$. A key fact for a regular theory $\T$ is that the resulting 2-functor $\Tmod (-)$ on regular categories is representable by the so called syntactic regular category (see \cite[Th. 6.5]{Bu}). 

\subsection{Syntactic categories and sites}\label{syncat}
Recall that for any first-order, geometric or regular theory $\T$ we get a \emph{\it syntactic category} $\cC_{\T}^\dag$ where the decoration $\dag = $ fo, gm or reg stands for first-order, geometric or regular respectively. All these (essentially small) categories $\cC_{\T}^\dag$ have objects the $\alpha$-equivalence classes of formulae over the signature and arrows $\T$-provable-equivalence classes of formulae which are $\T$-provably functional (see \cite[D1.4]{El}, \cite[2.4]{BJ} and \cite[\S 6]{Bu}) There is a Grothendieck (subcanonical) topology $J_{\T}^\dag$ on $\cC_{\T}^\dag$ (see \cite[D3.1]{El} and \cite[\S 3]{BJ}). For $\dag=$ reg the regular syntactic category $\cC_{\T}^\reg$ is provided with the topology $J^\reg$ (as it is any regular category) where a covering is given by a regular epi. For a site $(\cC, J)$ and any relevant category $\cD$ we shall denote by $${\sf  Lex}_{J}(\cC,\cD)$$ the corresponding left exact $J$-continuous functors, \ie sending $J$-covering sieves to epimorphic families. Thus ${\sf  Lex}_{J^\reg}(\cC_{\T}^\reg,\cC)$ is the category of regular functors from $\cC_{\T}^\reg$ to $\cC$ a regular category. 

\begin{lemma}\label{SGAlemma} Let $\cE=\Shv (\cC ,J)\subset \hat{\cC}$ be such that $\cC$ has finite limits and $J$ is subcanonical. For any topos $\cF$ we have natural equivalences 
$$\begin{array}{lll}
\text{\sf  Hom} (\cF, \hat{\cC})&\longby{\simeq}&\text{\sf  Lex}(\cC,\cF)\hspace*{14pt}\\
\hspace*{14pt}\bigcup & &\hspace*{4pt}\bigcup\\
\text{\sf  Hom} (\cF, \Shv (\cC ,J))&\longby{\simeq}&\text{\sf  Lex}_J(\cC,\cF)\hspace*{14pt}
\end{array}$$
where $f=(f_*,f^*)\leadsto f^*$ is sending a geometric morphism to the restriction of $f^*$ to $\cC$ regarded as a (full) subcategory of $\Shv (\cC ,J)$ \emph{via} the Yoneda embedding.
\end{lemma}
\begin{proof} This was proven by Grothendieck \& Verdier, see \cite[IV Cor. 1.7]{SGA4} \cf \cite[C2.3.9]{El}.
\end{proof} 

\subsection{Classifying topos}\label{Scltop} We shall denote $\cE [\T]\df\Shv (\cC_{\T}^\dag,J_{\T}^\dag)$ the corresponding topos dropping the refernce to $\dag$ when unnecessary. Recall the following:
\begin{thm}\label{cltop} Let $\cF$ be a topos. If $\T$ is geometric or regular we have natural equivalences 
$$ \Tmod (\cF)\longby{\simeq} \text{\sf Lex}_{J_{\T}^\dag}(\cC_{\T}^\dag,\cF) \longyb{\simeq} \text{\sf Hom} (\cF, \cE [\T])$$
for $\dag =$ {\rm gm or reg} respectively.
If $\T$ is a first-order theory we have a fully faithful functor 
$$\Tmod (\cF)\into \text{\sf Hom} (\cF, \cE [\T])$$
and there is a theory $\bar{\T}$ which is a (conservative) extension of $\T$ such that  $$\TTmod (\cF)\longby{\simeq}  \text{\sf Open} (\cF, \cE [\T])$$
where $\text{\sf Open}\subseteq \text{\sf Hom}$ is the full subcategory of open geometric morphisms.
\end{thm}
\begin{proof} Note that by the Lemma~\ref{SGAlemma} and inspecting $J^\dag$-continuous for $\dag =$ reg or gm one is left to show that the syntactic category $\cC_{\T}^\dag$ represents the $2$-functor $\Tmod (-)$. In fact, for a model $M\in \Tmod (\cF)$ we get a functor $F_M: \cC_{\T}^\dag\to \cF$ sending a formula to its interpretation and $F_M$ is left exact $J^\dag$-continuous: this yields the claimed equivalence.  For details see \cite[Th. 6.5]{Bu}, \cite[Th. X.6.1]{McM} and \cite[Th. D3.1.4]{El}. 

If $\T$ is first-order the claims are contained in \cite{BJ}. Recall that a geometric morphism $f: \cF \to\cE [\T]$ is open if the inverse image $f^*$ is Heyting (see \cite{J} and \cite[\S 1]{BJ}). Thus the fully faithful embedding is clear since $\Tmod (\cF)$ are just Heyting functors from $\cC_{\T}^{\rm fo}$ to $\cF$ (see \cite[2.4 \& (5)]{BJ}). Finally, such a theory $\bar{\T}$ of open functors is constructed in \cite[4.4]{BJ}. 
\end{proof}
 \begin{remark}
A topos $\cE $ is Boolean if and only if $$\text{\sf Open} (\cF, \cE)=\text{\sf Hom} (\cF, \cE)$$ for any topos $\cF$ (see \cite[3.5]{J}). Thus $\cE [\T]$ Boolean implies that $\T$ is Morita equivalent to $\bar{\T}$. Conversely,  if $\T$ is a geometric or regular theory Morita equivalent to $\bar{\T}$ then $\cE [\T]$ is Boolean. These are simple consequences of Theorem~\ref{cltop} (see also \cite[6.3]{BJ}). Actually, to get a Boolean classifying topos is quite restrictive (see \cite[D3.4]{El}). 
\end{remark}
\begin{defn} Call $\cE [\T]\df \Shv (\cC_{\T}^\dag,J_{\T}^\dag)$ the \emph{\it classifying topos} of the first-order theory $\bar{\T}$ and the  geometric/regular theory $\T$ respectively. 
\end{defn}
Models of the theory in the category of sets are points of the classifying topos: models of  $\T$ in the category $\cS$ of sets are given by the Theorem~\ref{cltop} as follows $$\Tmod (\cS)\into \text{Hom} (\cS, \cE [\T])\df \text{Points}(\cE [\T])\ \ M\leadsto f$$ 
Moreover, for $\T$ geometric/regular, the Theorem~\ref{cltop} yields a \emph{\it universal model} $M^{\T}\in \Tmod (\cE [\T])$ corresponding to the identity such that for any $M\in \Tmod (\cF)$  we have that 
\begin{equation}
f^*(M^{\T})=M
\end{equation}
for $f:  \cF \to\cE [\T]$ corresponding to $M$.  

\begin{remark} Every topos is the classifying topos of a theory. The geometric first-order theory $\T$ of a topos $\cE=\Shv (\cC ,J)$ is the theory of left exact $J$-continuous functors on the signature given by $\cC$ (see \cite[D3.1.13]{El}). The Lemma~\ref{SGAlemma} shows us that $\cE \cong \cE [\T]$. The theory $\T$ is regular if and only if $J$ is generated by singleton covering families (see \cite[D3.3.1]{El}).  There is also a first-order theory $\bar{\T}$ whose models are inverse images of open geometric morphisms (see \cite[\S 4]{BJ}).
 \end{remark}

\section{A regular (co)homology theory}\label{rehomth}
All the axiomatizations of a homology theory are quite involved but possibly first-order.  They are starting with a given category $\cC$ which we may consider the indexing set for our signature $\Sigma$ containing it as sorts and function symbols with additional symbols formalizing an algebraic structure, \eg abelian groups, for each sort.  The minimal reasonable axiomatization we can imagine is provided by a set of regular sequents defining a regular (mixed) homology theory $\T$ as follows.   

\subsection{Category of pairs}\label{pairs}
Let $\cC$ be any category and fix  a subcategory $\cM$ of $\cC$.  
Let $\cC^{\square}$ be the category with objects the arrows in $\cM$ and morphisms the commutative squares of $\cC$. Call $\cC^{\square}$ a \emph{category of pairs} and denote $(X,Y)$ an object of $\cC^{\square}$, \ie  a morphism $f: Y \to X$ of $\cM$, and $\square: (X,Y) \to (X', Y')$ a commutative square
$$\xymatrix{Y\ar[r]^-{f}  \ar[d]_-{} & X \ar[d]_-{} \\
Y' \ar@{>}[r]^-{f'} & X'}$$
We shall denote $\boxdot : (X,Y)\to (X,Y)$ the identity and $\boxminus : (X,Y) \to (X'', Y'')$ the composition of $\square: (X,Y) \to (X', Y')$ and $\square': (X',Y') \to (X'', Y'')$. Moreover, for $f : Z \to Y$ and $g: Y \to X$ objects of $\cC^{\square}$ we shall denote $\partial  : (Y,Z) \to (X,Y)$ the following morphism of $\cC^{\square}$ 
$$\xymatrix{Z\ar[r]^-{f}  \ar[d]_-{f} & Y \ar[d]^-{g} \\
Y \ar@{>}[r]^-{g} & X}$$
Note that $\partial$ has a canonical factorization in $\cC^{\square}$ given by 
$$
\xymatrix{Z\ar[r]^-{f}  \ar[d]_-{||} & Y \ar[d]^-{g} \\
Z \ar@{>}[r]^-{gf}\ar[d]_-{f}  & X\ar[d]_-{||}\\
Y \ar@{>}[r]^-{g} & X}
$$
and we let $\boxtimes: (Y,Z) \to (X,Z)$ and $\boxplus:(X,Z)  \to (X,Y)$ be the above squares. Finally, we  shall call a 
 $\partial$-cube of $\cC$ the following commutative square 
\begin{equation}\label{dcube}
\xymatrix{(Y,Z)\ar[r]^-{\partial}  \ar[d]_-{\Diamond} & (X,Y) \ar[d]^-{\square} \\
(Y',Z')\ar@{>}[r]^-{\partial'} & (X',Y')}
\end{equation}
of $\cC^{\square}$ induced by a pair of arrows $\Diamond: (Y,Z)\to (Y',Z')$ and $\square : (X,Y)\to (X',Y')$.

\subsection{Signature}\label{signature}
 Consider the following signature $\Sigma$ based on a category of pairs. Let $h_n(X,Y)$ denote sorts and variables $x:h_n(X,Y)$ indexed by $n\in \Z$ and all objects of $\cC^{\square}$. Consider $+$ a binary function symbol $h_n(X,Y)^2\to h_n(X,Y)$, an 1-ary function symbol $h_n(X,Y)\to h_n(X,Y)$ with value $-x$ for $x:h_n(X,Y)$ and a constant $0$ of sort $h_n(X,Y)$ indexed by $n\in \Z$ and all objects of $\cC^{\square}$. Consider function symbols $\square_n: h_n(X,Y) \to h_n(X',Y')$ corresponding to morphisms of $\cC^{\square}$ and an additional function symbol $\partial_n :h_n(X,Y) \to h_{n-1}(Y,Z)$ corresponding to $\partial : (Y,Z) \to (X,Y)$ morphism of $\cC^{\square}$.
No relation symbols apart from the equality.
 
\subsection{Axioms}\label{axiom} For each sort $h_n(X,Y)$, \ie for each $(X,Y)$ object object of $\cC^{\square}$ and $n\in \Z$, we introduce the following axioms:
\begin{itemize}
\item[$\sqcup$1] $h_n(X,Y)$ is a group, \ie
$$\top \vdash_{x,y,z} (x+y)+z = x+(y+z)$$ 
$$\top\vdash_x 0+x = x+0 =x$$
$$\top\vdash_x (x+(-x)=0)\wedge ((-x)+ x = 0)$$ 
and abelian if 
$$\top\vdash_{x,y} x+ y = y+x$$ 

\item[$\sqcup$2]  $\square_n :h_n(X,Y) \to h_n(X',Y')$ and $\partial_n : h_n(X,Y) \to h_{n-1}(Y,Z)$ are group homomorphisms, \ie
$$\top\vdash_{x,y} \square_n(x+y) = \square_n(x)+\square_n(y)$$ 
$$\top\vdash_{x,y} \partial_n(x+y) = \partial_n(x)+\partial_n(y)$$

\item[$\sqcup$3]  $h_n(X,Y)$ are functors on $\cC^{\square}$ and $h_n(X,Y) \to h_{n-1}(Y,Z)$ is natural, \ie given $\square_n :h_n(X,Y) \to h_n(X',Y')$, $\square_n' :h_n(X',Y') \to h_n(X'',Y'')$ and $\boxminus_n :h_n(X,Y) \to h_n(X'',Y'')$
$$\top\vdash_x \square_n'(\square_n(x)) = \boxminus_n(x)$$
and for $\boxdot_n :h_n(X,Y)\to h_n(X,Y)$
$$\top\vdash_x \boxdot_n (x)=x$$
and associated to a $\partial$-cube \eqref{dcube}
$$\top\vdash_x \Diamond_n(\partial_n(x)) = \partial'(\square_n(x))$$

\item[$\sqcup$4]  $h_n(Y,Z) \to h_n(X,Z) \to h_n(X,Y) \to h_{n-1}(Y,Z)\to h_{n-1}(X,Z)$ is exact, \ie it is a complex
$$\top\vdash_x \boxplus_n(\boxtimes_n(x)) =0$$
$$\top\vdash_x \partial_n(\boxplus_n(x)) =0$$
$$\top\vdash_x \boxtimes_{n-1}(\partial_n(x)) =0$$
and finally
$$\boxplus_n(x)=0\vdash_x  (\exists y) (\boxtimes_n(y) =x)$$
$$\partial_n(x)=0\vdash_x  (\exists y) (\boxplus_n(y) =x)$$
$$\boxtimes_{n-1}(x)=0\vdash_x  (\exists y) (\partial_n(y) =x)$$\\
\end{itemize}

Summarizing \ref{pairs}, \ref{signature} and \ref{axiom} we set:
\begin{defn}\label{homtop} Call (mixed) \emph{homology theory $\T$} the regular theory with the sequents $\sqcup$1-4 over the signature $\Sigma$ associated to $\cC^{\square}$ a category of pairs. 
\end{defn}

Let $\cE [\T]$ be the corresponding classifying topos, \ie the \emph{motivic topos}. We have:\begin{lemma}\label{regcon} $\cE [\T]$ is connected and locally connected.
\end{lemma} 
\begin{proof} In fact, the regular syntactic site $(\cC_{\T}^{\rm reg},J^{\rm reg})$ is a locally connected site whose underlying category has a terminal object and one just applies \cite[C3.3.10]{El}.\end{proof}

Note that if we denote $h^n(X,Y)$ the sorts in \ref{signature} and consider function symbols $\square^n: h^n(X',Y') \to h^n(X,Y)$ corresponding to morphisms $\square : (X,Y)\to (X',Y')$ of $\cC^{\square}$   and an additional function symbol $\partial^n :h^{n}(Y,Z)\to h^{n+1}(X,Y)$ corresponding to $\partial : (Y,Z) \to (X,Y)$ we get a signature that we denote $\Sigma^{op}$. Reversing the arrows in the axioms \ref{axiom} we get a corresponding list of sequents $\sqcap$1-4. Rewriting \ref{axiom} over $\Sigma^{op}$ we set:

\begin{defn}\label{cohomtop} 
Call (mixed) \emph{cohomology theory} $\T^{op}$ the regular theory with the list of sequents $\sqcap$1-4 over the signature $\Sigma^{op}$. 
\end{defn}

\begin{remark}\label{resign} Note that for any category $\cC'$ together with a pair of compatible forgetful functors $\cC'\to \cC$ and $\cM'\to \cM$ the theory $\T$ (resp. $\T^{op}$) on the signature $\Sigma$ (resp. $\Sigma^{op}$) given by $(\cC, \cM)$ can be regarded as well on the restricted signature $\Sigma'$ (resp. $\Sigma'^{op}$) given by $(\cC', \cM')$.
\end{remark}

\subsection{Further axioms} We may and will add regular, geometric and first-order axioms to the above regular theory as soon as we enrich the category of pairs with further structure: several canonical properties  can be easily axiomatized in first-order logic and any (co)homology theory $\T^\prime$ should appear as an extension of $\T$.  We have:
\begin{lemma} \label{fadd}
Let $\T^\prime$ be a regular or geometric theory over the same signature $\Sigma$ obtained from $\T$ by adding axioms. We then get $f = (f_*, f^*): \cE[\T^\prime] \into \cE[\T]$ which is an embedding.   In particular$$f^*: \Ab(\cE [\T])\onto \Ab(\cE [\T^\prime])$$ is a Serre quotient with a section
$$f_*: \Ab(\cE [\T^\prime])\into \Ab(\cE [\T])$$ which is fully faithful. 
\end{lemma}
\begin{proof}  Let $\cE [\T]$ be $\Shv (\cC_{\T}^\dag,J_{\T}^\dag)$ for $\dag$ = reg or gm and similarly for $\T^\prime$. By arguing as in the proof of the duality theorem of \cite{C} one can see that $\cE [\T^\prime]$ is a subtopos of $\cE [\T]$ via Theorem \ref{cltop}. Let $J_{\T^\prime}^{\T}$ be the associated $\T$-topology of $\T^\prime$, \ie  the smallest Grothendiek topology on $\cC_{\T}^\dag$ generated by  all the $J^\dag_\T$-covering sieves and the sieves containing a morphism corresponding to an axiom of $\T^\prime$. 
Thus a functor $\cC_{\T}^\dag\to \cF$ to a topos $\cF$ is left exact and $J_{\T^\prime}^{\T}$-continuous if and only if it is a model of $\T^\prime$. Using Lemma \ref{SGAlemma} then $\Shv (\cC_{\T}^\dag,J_{\T^\prime}^{\T})$ is the classifying topos of $\T^\prime$. Thus $$f= (f_*, f^*): \cE[\T^\prime]=\Shv (\cC_{\T^\prime}^\dag,J_{\T^\prime}^\dag)\cong \Shv (\cC_{\T}^\dag,J_{\T^\prime}^{\T}) \subset\Shv (\cC_{\T}^\dag,J_{\T}^\dag)=\cE [\T]$$
is an embedding, \ie $\cE[\T^\prime]$ is a subtopos of $\cE[\T]$. 
That means $f^*f_* \cong id$.  Thus the exact functor  $f^*: \Ab(\cE [\T])\to \Ab(\cE [\T^\prime])$ is essentially surjective,  $\Ab(\cE [\T^\prime])$ is the Serre quotient of $\Ab(\cE [\T])$ by $\ker f^*$ and $f_*$ is fully faithful. \end{proof}
\begin{remark} Note that $\cE [\T^\prime]$ is a localization of $\cE [\T]$ at the class of morphisms sent by $f^*$ to isomorphisms in $\cE [\T^\prime]$ and the corresponding local objects of $\cE [\T]$ are precisely the objects of $\cE [\T^\prime]$. Similarly, for $\Ab(\cE [\T^\prime])$ and $\Ab(\cE [\T])$ (\cf \cite[A.2.10]{NT})\end{remark}

\subsection{$I^+$-invariance and $cd$-exactness}\label{hinv} Assuming that $\cC$ has finite products let $1$ be the final object. Recall that Voevodsky \cite[\S 2.2]{Vh} call an interval of $\cC$ an object $I^+$ together with morphisms $m:I^+\times I^+\to I^+$ and $i_0,i_1:1\to I^+$ such that 
\begin{equation}\label{eqio}
m(i_0\times id) = m(id\times i_0) = i_0p\hspace*{0.5cm}
m(i_1\times id) = m(id\times i_1) = id
\end{equation}
where $p:I^+\to 1$ is the canonical morphism, $id$ is the identity of $I^+$ and $i_0 \times id:I^+\cong 1\times I^+\to I^+\times I^+$, etc. In general, consider and denote $i_k^X: X\cong X\times 1\to X\times I^+$ the morphism induced by $i_k$ for $k=0,1$ and  assume that
\begin{equation}\label{phi}
\xymatrix{Y\ar[r]^-{f}  \ar[d]_-{i_k^Y} & X \ar[d]^-{i_k^X} \\
Y\times I^+ \ar@{>}[r]^-{f\times id} & X\times I^+}
\end{equation}
are morphisms in $\cC^{\square}$ that we denote $\square^+_k:(X,Y)\to (X\times I^+,Y\times I^+)$ for $k=0, 1$.
\begin{defn}\label{homtopinv} Say that the homology theory $\T$ in \ref{homtop} is \emph{$I^+$-invariant} if additionally 
\begin{itemize}
\item[$\sqcup$5] $\square^+_{0,n}=\square^+_{1,n}:h_n(X,Y)\to h_n(X\times I^+,Y\times I^+)$, \ie we have that $$\top\vdash_x \square^+_{0,n}(x)=\square^+_{1,n}(x)$$
\end{itemize}
Denote $\T^+$ the resulting homotopy invariant regular theory on the same signature of $\T$.
\end{defn}

Further assume that $\cC$ has an initial object $\emptyset$. Assuming that $\emptyset \to X$ is in $\cM$ for each object $X$ of $\cC$ we denote $$h_n(X)\df h_n(X,\emptyset)$$
the associated sorts for $n\in \Z$. Assume that $\cC$ or a category $\cC'_{/X}$ of triangles over an object $X$ of $\cC$ is endowed with a $cd$-structure in the sense of Voevodsky \cite{Vcd}. Recall that a $cd$-structure on a category is a class of distinguished commutative squares which is stable by isomorphism. There is a corresponding Grothendieck $cd$-topology $X_{cd}$ associated to $\cC'_{/X}$ but we just consider the $cd$-structure here. Note that in a category $\cC$ with coproducts we may take the $cd$-structure given by the squares
$$\xymatrix{\emptyset \ar[r]^-{}  \ar[d]_-{} & X \ar[d]^-{} \\
Y \ar@{>}[r]^-{} & Y\coprod X}$$ 
that we call the coproduct $cd$-structure (and similarly for push-outs). In general, for a given $cd$-structure, considering each distinguished square 
$$\xymatrix{B\ar[r]^-{}  \ar[d]_-{} & C \ar[d]^-{} \\
A \ar@{>}[r]^-{} & D}$$ 
of $\cC'_{/X}$ we have function symbols in our signature $\Sigma$ in \ref{signature} for each side of the square but we need to enlarge $\Sigma$ by adding a function symbol $h_n(D)\to h_{n-1}(B)$, for each $n\in \Z$.
\begin{defn}\label{cdinv} Say that the homology theory $\T$ in \ref{homtop} is \emph{$cd$-exact} if for each $n\in \Z$ and each distinguished square we have that the canonical Mayer-Vietoris sequence
$$\cdots \to h_n(B)\to h_n(A)\times h_n(C) \to h_n(D)\to h_{n-1}(B)\to\cdots$$
\begin{itemize}
\item[$\sqcup$6] is exact and 
\item[$\sqcup$7] it is natural with respect to morphisms of distinguished squares.
\end{itemize} 
\end{defn}
Similarly to the above, as an exercise for the interested reader, one can express in a more formal way the axioms for $cd$-exactness. 
\section{Models of (co)homology theories}\label{models}
Consider a given $\cC$ and $\cC^{\square}$ a category of pairs and let $H\in \Tmod(\cD)$ be a model (of the homology theory $\T$ defined in \ref{homtop}) in a regular category $\cD$.
\subsection{Basic properties}We then have, for each morphism $Y \to X$ in $\cM$ and $n\in \Z$, an internal (abelian) group $H_n(X,Y)\in \cD$ which depends functorially on pairs, \ie we have a functor 
$$(X,Y) \leadsto H_*(X,Y)\df \{H_n(X,Y)\}_{n\in \Z}$$
from $\cC^{\square}$ to $\Z$-families of internal (abelian) groups in $\cD$.  For a $\partial$-morphism given by $Z\to Y\to X$ in $\cM$ we have a long exact sequence of (abelian) groups in $\cD$ 
$$\cdots\to H_n (Y,Z)\to H_n (X,Z)\to H_n (X,Y)\to H_{n-1} (Y,Z)\to \cdots$$
which is natural with respect to $\partial$-cubes \eqref{dcube} by \ref{axiom}~$\sqcup$3-4.

Dually, $H\in \Top(\cD)$ consists of a contravariant functor 
$$(X,Y) \leadsto H^*(X,Y)\df \{H^n(X,Y)\}_{n\in \Z}$$
from $\cC^{\square}$ to $\Z$-families of internal (abelian) groups in $\cD$ such that for a $\partial$-morphism given by $Z\to Y\to X$ in $\cM$ we have 
$$\cdots\to H^n (X,Y)\to H^n (X,Z)\to H^n (Y,Z)\to H^{n+1} (X,Y)\to \cdots$$
which is natural with respect to $\partial$-cubes. 
\begin{lemma}\label{purity}
Let $\cC$ and $\cM$ be as above. If $H\in \Tmod (\cD)$ (resp.  $H\in \Top(\cD)$) then $H_*(X, Y) =0$ (resp. $H^*(X, Y) =0$) for all $Y \cong X$ isomorphisms in $\cM$.
\end{lemma}
\begin{proof} Note that an isomorphism $(X,Y) \cong (X,Y')$ such that 
$$\xymatrix{Y\ar[r]^-{f}  \ar[d]_-{\iota} & X \ar[d]^-{||} \\
Y' \ar@{>}[r]^-{g} & X}$$
with $\iota$ isomorphism of $\cC$ yields $H_*(X,Y)\cong H_*(X,Y')$ by functoriality. If $g$ is also an iso of $\cM$,  $\iota = g^{-1}$ and $f=id_X$ we get $H_*(X,Y')\cong H_*(X,X)$. Now from the exactness of 
$$H_*(X,X) \by{id} H_*(X,X) \by{id} H_*(X,X)$$
for $X=Y=Z$ we obtain $H_*(X,X) =0$.\end{proof}

Further assume that $\cC$ has an initial object $\emptyset$ and $\emptyset \to X$ is in $\cM$ for each $X$ object of $\cC$. We may wish that if $X\to \emptyset$ is in $\cM$ then $X\cong\emptyset$ (for example, if $\cM =\cC$ that means $\emptyset$ strictly initial).
Denote 
\begin{equation}\label{zero}
H_*(X,\emptyset) \df H_*(X)
\end{equation}
and note that $H_*(\emptyset) =0$. Moreover, consider composable arrows
$$W\to Z\to Y\to X$$
in the category $\cM$ providing a pair of $\partial$-morphims.
\begin{lemma}\label{2=0}
 Let $H\in \Tmod (\cE)$ be any model of $\T$ in a regular category $\cE$.
If $(Z,W)\to (Y, Z)\to (X,Y)$ are $\partial$-morphisms then $\partial_*^2=0$, \ie the composition
$$H_*(X,Y)\to H_{*-1}(Y, Z)\to H_{*-2}(Z,W)$$
is the zero morphism of abelian group objects.
\end{lemma}
\begin{proof} Consider the following $\partial$-cubes \eqref{dcube} of $\cC^\square$
$$
\xymatrix{(Y,\emptyset )\ar[r]^-{\partial}  \ar[d]_-{} & (X,Y) \ar[d]^-{||} \\
(Y,Z)\ar@{>}[r]^-{\partial } & (X,Y)} \hspace{2cm}
\xymatrix{(Z,\emptyset )\ar[r]^-{\partial}  \ar[d]_-{} & (Y,Z) \ar[d]^-{||} \\
(Z,W)\ar@{>}[r]^-{\partial } & (Y, Z)}
$$
By naturality we get the following commutative squares
$$
\xymatrix{H_*(X,Y)\ar[r]^-{\partial_*}  \ar[d]_-{||} & H_{*-1} (Y) \ar[d]^-{} & \\
H_*(X,Y)\ar[r]^-{\partial_*} & H_{*-1}(Y,Z)\ar[r]^-{\partial_*}  &  H_{*-2}(Z,W)\\
& H_{*-1}(Y,Z)\ar[u]_-{||} \ar[r]^-{\partial_*} & H_{*-2}(Z)\ar[u]}
$$
By exactness we then get that $\partial_*^2=0$ as claimed.
\end{proof}

For example, in all cases listed below but $\cM=\cC$, we consider the subcategory $\cM$ with the same objects of $\cC$ and such that a morphism $f$ is in $\cM$ if $f$ is a mono or is in a distinguished class of monos including all isomorphisms of $\cC$. Also note that if $\cM$ are just isomorphisms of $\cC$ then $H_*(X)$ are the only possibly non-zero homology groups. 

\subsection{Grothendieck exact $\partial$-functors} Let $\cC$ and $\cD$ be abelian categories. Let $\cM$ be given by the monos of $\cC$. For an exact covariant homological $\partial$-functor  $T_n$ from $\cC$ to $\cD$ let $$H_n(X,Y)\df T_n(X/Y)$$ Thus $H\in \Tmod (\cD)$ from the definition of $\partial$-functor (see \cite[2.1.1]{WH}). In fact, given $Z\into Y\into X$ we clearly have a short exact sequence 
$$0\to Y/Z \to X/Z \to X/Y\to 0$$
and therefore a long exact sequence 
$$\cdots \to T_n(Y/Z)\to T_n(X/Z)\to T_n(X/Y)\to T_{n-1}(Y/Z)\to \cdots$$
which is natural with respect to $\partial$-cubes so that all axioms $\sqcup$1-4 are satisfied. Furthermore $H_n(X, Y)=T_n(X/Y)=0$ for $n<0$. For an exact contravariant cohomological $\partial$-functor $T^n$ from $\cC$ to $\cD$ setting $H^n(X,Y)\df T^n(X/Y)$ we get $H\in \Top (\cD)$.

\subsection{Barr-Beck homology}  Let ${\sf G} =(G, \varepsilon, \delta)$ be a cotriple in $\cC$. Let $E:\cC\to \cD$ be a functor where $\cD$ is an abelian category and $\cM= \cC$. Let $$H_n(X,Y)\df H_n(Y\to X,E)_{\sf G}$$
be the relative Barr-Beck homology with coefficients in $E$ with respect to the cotriple ${\sf G}$ (see \cite[8.7.1]{WH}). Here also $H_n(X,Y)=0$ for $n<0$. Since this relative homology is given by a cone construction is clear that $H\in \Tmod (\cD)$.

\subsection{Connes cyclic homology} We let $\cC$ be the abelian category of cyclic objects in an  abelian category $\cA$ and let $\cM$ be the monos of $\cC$ (see \cite[9.6.4]{WH}). For $Y\into X$  let $$HC_n(X,Y)\df \Tot CC_{\d \d}(X/Y)$$ 
be cyclic homology where $CC_{\d \d}$ is the Tsygan double complex (see \cite[9.6.6-7]{WH}). Then $HC\in \Tmod (\cA)$.
 
\subsection{Tate cohomology} Let $\cC =\Gmod$ be the abelian category of 
$G$-modules with $G$ finite and $\cD= \Ab$. Let $\cM$ be the monos of $\cC$ and set
$$H_{-n}(X,Y) \df \hat{H}^{n}(G, X/Y)$$
the Tate cohomology of the $G$-module $X/Y$. We have $H\in \Tmod (\Ab)$ and $H_n(X, Y) \neq 0$ for $n \in \Z$. 

\subsection{Singular (co)homology}\label{singhom}
For $\cC = {\rm Top}$ the category of topological spaces and $\cM$ the subcategory of topological embeddings let ${\rm Sing}_\d(X)$ be the singular chain complex. Then $f: Y\into X\in \cM$ induces an inclusion ${\rm Sing}_\d(Y)\into {\rm Sing}_\d(X)$ so that 
$$H_{n}^{\rm sing}(X,Y) \df H_{n}({\rm Sing}_\d(X)/{\rm Sing}_\d(Y))$$
yields a model $H^{\rm sing}\in \Tmod (\cS)$ as it is well known that the axioms $\sqcup$1-4 are satisfied in this case. 
Similarly, for $\cC=\cM$ the category of simplicial topological spaces and $H^*(X,Y)$ the relative cohomology of the constant sheaf $\Z$ the axioms $\sqcap$1-4 are satisfied. 

For $\cC' =\Sch_k$ the category of $k$-algebraic schemes where $k\into \C$ is a subfield of the complex numbers and $\cM'$ the subcategory of closed embeddings we have the forgetful functor of $\C$-points $f:Y \to X\leadsto Y(\C)\to X(\C)$ from $\Sch_k$ to ${\rm Top}$ sending a closed embedding to a closed subspace. We may thus consider the restricted signature (as in Remark \ref{resign}) whence
$$(X, Y)\leadsto \{H_{n}^{\rm sing}(X (\C),Y (\C))\}_{n\in \Z}$$ is a model $H^{\rm sing}\in \Tmod (\cS)$ over the resctricted signature as well. Dually, we have $H_{\rm sing}\in \Top (\cS)$ for singular cohomology.

\subsection{Algebraic singular homology} \label{algsinghom}
Let $\cC = \Sch_k$ be the category of $k$-algebraic schemes and $\cM$ closed subschemes (with the reduced structure). Let $\cD = {\rm PST}$ be the abelian category of additive presheaves of abelian groups on the additive category $\Cor_k$ of Voeovodsky finite correspondences (see \cite[Def. 2.1 \& Th. 2.3]{VL}). For $X\in \Sch_k$ let $\Z_{tr}(X)\in {\rm PST}$ be the representable presheaf $U\leadsto \Cor_k(U, X)$ for $U$ smooth $k$-algebraic scheme (\cf \cite[2.11]{VL}). For $F\in {\rm PST}$ let $C_n (F)(U)\df F(U\times_k \Delta^n)$ where $\Delta^n = \Spec (k[t_0, \ldots , t_n]/(t_0+\cdots+t_n -1))$ is a cosimplicial $k$-scheme so that we obtain a chain complex 
$$\cdots \to F(U\times_k \Delta^2)\to F(U\times_k \Delta)\to F(U)\to 0$$
yielding $C_\d (F)\in {\rm PST}$ (see \cite[2.14]{VL}). For the Suslin-Voevodsky singular chain complex  ${\rm Sing}_\d^{\rm SV}  (X)\df C_\d (\Z_{tr}(X))$  we  have that ${\rm Sing}_\d^{\rm SV}  (Y)\into {\rm Sing}_\d^{\rm SV}  (X)$ if $Y\into X$ is a closed subscheme so that 
$$H_{n}^{\rm SV}(X,Y) \df H_{n}({\rm Sing}_\d^{\rm SV}(X)/{\rm Sing}_\d^{\rm SV}(Y))$$
yields a model $H^{\rm SV}\in \Tmod ({\rm PST})$ with a similar proof as for classical singular homology.

\subsection{Homotopy and Mayer-Vietoris} If $\cC$ is provided with an interval object $I^+$ (with the notation adopted in \ref{hinv}) define an \emph{$I^+$-homotopy} between two parallel maps $\square^0$ and $\square^1$ from $(X,Y)$ to $(X',Y')$ as usual via a morphism $$(X\times I^+,Y\times I^+)\to (X',Y')$$ and a factorization through $\square^+_k:(X,Y)\to (X\times I^+,Y\times I^+)$ for $k=0, 1$ respectively.  If $H$ is a model of the $I^+$-invariant theory $\T^+$ in \ref{homtopinv} then $I^+$-homotopy maps induce the same map $$\square^0_*=\square^1_*:H_*(X,Y)\to H_*(X',Y')$$ and furthermore we have the following standard fact:
\begin{lemma} $H\in\Tmod(\cD)$ is a model of $\T^+$ if and only if 
$$\Pi_*:H_*(X\times I^+,Y\times I^+)\longby{\simeq} H_*(X,Y)$$
where $\Pi: (X\times I^+,Y\times I^+)\to (X,Y)$ is the canonical projection.
\end{lemma}
\begin{proof} The projection $\Pi$ is such that $\Pi \square^+_k = id_{(X,Y)}$ for both $k=0,1$ in \eqref{phi}. Thus $\Pi_*$ iso implies
$\square^+_{0,*} = \square^+_{1,*}$. Conversely, to see that $\square^+_{0,*}\Pi_* = id$ one can make use of the induced  relations \eqref{eqio} by taking product with $f:Y\to X$. Denoting by $\Lambda$ the following morphism 
$$\xymatrix{Y\times I^+\times I^+\ar[r]^-{f\times id}  \ar[d]_-{id\times m} & X\times I^+\times I^+ \ar[d]^-{id\times m} \\
Y\times I^+ \ar@{>}[r]^-{f\times id} & X\times I^+}
$$
of $\cC^{\square}$ we get that $\Lambda \square^+_{0}=\square^+_{0}\Pi$ and $\Lambda \square^+_{1}=id_{(X,Y)}$. Thus 
$\Lambda_* \square^+_{0,*}= \Lambda_* \square^+_{1,*} = \square^+_{0,*}\Pi_* = id$ if $H$ is a $\T^+$-model.
\end{proof}
Both singular homology $H^{\rm sing}$ and algebraic singular homology $H^{\rm SV}$ are models of $\T^+$ by taking $I^+$ the real interval $[0, 1]$ and the affine line $\Aff^1_k$ respectively: in both cases $m$ is the multiplication, $i_0$ and $i_1$ are the (rational) points $0$ and $1$ (\cf \cite[2.19]{VL}). 

Furthermore, they are $cd$-exact with respect to several $cd$-structures. For example, by considering the $cd$-structure given by the following squares over a fixed $X$  
$$\xymatrix{U\cap V\ar[r]^-{}  \ar[d]_-{} & V \ar[d]^-{} \\
U \ar@{>}[r]^-{} & U\cup V}$$ 
where $U\into X$ and $V\into X$ are in the category $\cC'_{/X}$ of open embeddings for $\cC = {\rm Top}$ and of Zariski open for $\cC = \Sch_k$ we get that both $H^{\rm sing}$ and $H^{\rm SV}$ satisfies $cd$-exactness: this is a reformulation of the usual Mayer-Vietoris long exact sequences.  
\begin{lemma} Any $\T$-model $H$ in an abelian category is $cd$-exact with respect to the $cd$-structure given by those commutative squares of $\cM$
$$\xymatrix{Y\ar[r]^-{}  \ar[d]_-{} & X_1 \ar[d]^-{} \\
X_2\ar@{>}[r]^-{} & X}$$ 
such that $$H_*(X_1,Y)\oplus H_*(X_2,Y)\longby{\simeq} H_*(X,Y)$$
\end{lemma}
\begin{proof} The proof is an easy diagram chase: for example, it is exactly the same for proving Mayer-Vietoris Theorem for Barr-Beck homology.\end{proof}

\subsection{Weil first-order (co)homology theories}
A general question arising from the above and addressed to the experts in model theory is to put (co)homology theories in the framework of first-order (geometric/regular) theories which are extensions of $\T$ and to study the corresponding motivic topos.  Fragments of Grothendieck-Weil (co)homology theory and shadows of Grothendieck-Verdier duality are included. 

Let $\cC$ be the category of algebraic schemes over $k = \bar k$ with proper morphisms and let $\cM$ be the subcategory given by closed subschemes.  Then Borel-Moore \'etale homology $H^{\et}_*$ of pairs is a $\T$-model in sets (see \cite{Let} for properties of \'etale homology).

Recall that for a pair $(X,Y)$ and $U=X\setminus Y$ the open complement we have $H^{\et}_*(X,Y)= H^{\et}_*(U)$ which are finitely generated modules. Further $H_i^\et(X) = 0$ for $i< 0$ and $i > 2d$ where $d = \dim (X)$ and $H_0^\et(X)$ is the free module on the proper connected components while $H_{2d}^\et(X)$ is the free module on $d$-dimensional irreducible components. A K\"unneth formula holds for \'etale homology. For smooth projective algebraic varieties $X$ over $k$ we have a Poincar\'e duality isomorphism
$$H^{2d-i}_\et(X)\cong H_i^\et(X)$$ where $d\df \dim (X)$. Actually, we notably have a key result due to A. Macintyre \cite{M} justifying the following:
\begin{thm} Weil cohomologies are models of a first-order theory $\T_W^{op}$ which can be regarded as an extension of $\T^{op}$ and equivalent to Kleiman's axiomatization.
\end{thm}
\begin{proof} This is a reformulation of \cite[\S 3]{M}. The signature $\Sigma$ is enlarged by including a sort for each such a variety and a function symbol for a morphism and a coefficient field sort $K$ if we want to take care of coefficients.
\end{proof} 

As a consequence we can call \emph{Grothendieck-Weil topos} the motivic topos $\cE[\T_W^{op}]$.  However, there are weaker versions of Weil cohomologies considered by Y. Andr\'e \cite{An}: a \emph{pure} Weil cohomology (without Lefschetz) and a \emph{mixed} Weil cohomology, \ie
$H^*:\Cor_k^{op} \to \cA$
where $\cA$ is an abelian $\otimes$-category such that $H^*$ is homotopy invariant, verifies K\"unneth and Mayer-Vietoris.
Call \emph{Andr\'e-Voevodsky topos} its motivic topos.  Finally, Bloch-Ogus axiomatizations \cite{BO} and \cite{LBVBO} would yield the \emph{Bloch-Ogus topos}.

\section{Theoretical motives and motivic complexes}\label{motives}

Consider our regular (co)homology theory $\T$ on a base category $\cC$ along with a distinguished subcategory $\cM$.

\subsection{Constructible $\T$-motives} For abelian group sorts $h_n(X,Y)$ in the previously mentioned signature $\Sigma$ we have the following key fact.
\begin{lemma}\label{synadd} The syntactic category $\cC_{\T}^{\rm reg}$ is an additive category.\end{lemma}
\begin{proof} Similar (and in fact simpler) to \cite[Lemma~2.4]{BVCL}.
\end{proof}
\begin{defn} Denote $\cA[\T]$ the Barr exact completion of the syntactic category $\cC^{\rm reg}_\T$. Call $\cA[\T]$ the category of \emph{constructible effective $\T$-motives}.\end{defn}
\begin{propose}\label{abeliancon}
The category $\cA[\T]$ of constructbile $\T$-motives is an abelian category and
$$ \Tmod (\cE)\cong \text{\sf Ex} (\cA[\T], \cE)$$
for any $\cE$ Barr exact category, \eg an abelian category.
\end{propose}
\begin{proof}As $\cC_{\T}^{\rm reg}$ is an additive category by Lemma \ref{synadd} its Barr exact completion is also additive and thus abelian (by a well known theorem of M. Tierney abelian is equivalent to additive and Barr exact). Moreover, from Theorem~\ref{cltop}, any $\T$-model in $\cE$ exact is a left exact $J^{\rm reg}_\T$-continuous functor $\cC_\T^{\rm reg}$ to $\cE$ which is just a regular functor. Thus from the universal property of the exact completion (see \cite[Thm. 3.3]{La}) we have that it yields an exact functor $\cA[\T]\to \cE$. Conversely, note that $\cC_{\T}^{\rm reg}\into \cA[\T]$ is fully faithful and regular (see \cite[3.2]{La}) so that we can just use the same argument backwards.
\end{proof}
Note that $\cA[\T]$ is ``homological''. Similarly, we obtain $\cA[\T^{op}]$ the category of constructible effective $\T^{op}$-motives which is ``cohomological''. We have the following duality result.
\begin{propose} There is a canonical equivalence $$ \cA[\T^{op}] \cong \cA[\T]^{op}$$
\end{propose}
\begin{proof} If $\cA$ is abelian also $\cA^{op}$ is abelian. Any object of $\cA$ is endowed with a unique abelian group and co-group structure, \ie a group in $\cA^{op}$. Therefore we get
$$\Top (\cA)= \Tmod (\cA^{op})$$
By Proposition \ref{abeliancon} we then obtain
$$\text{\sf Ex} (\cA[\T^{op}], \cA)\cong \text{\sf Ex} (\cA[\T]^{op}, \cA)$$
Since this equivalence holds for any abelian category $\cA$ and it is natural with respect to the variable $\cA$ we get the claimed equivalence. 
\end{proof}
There is a universal model corresponding to the Yoneda embedding via Theorem~\ref{cltop} and the category $\cA[\T]\subset \cE[\T]$ is the full subcategory given by all coequalizers of equivalence relations in $\cC_\T^{\rm reg}$ (see \cite[\S 3]{La}). Therefore, using the Yoneda embedding 
$$\cC_\T^{\rm reg}\into \cA[\T]\subset \cE[\T]$$ 
and Proposition \ref{abeliancon} we set:

\begin{defn} \label{univhom}
Denote $H^\T\in \Tmod (\cA[\T])$ the \emph{universal homology} corresponding to the universal model of $\T$, \ie the identity of $\cA [\T]$. For $\cE$ Barr exact and $H\in \Tmod (\cE)$ we denote
$$r_H : \cA[\T]\to \cE$$
and call it the \emph{realization} functor associated to $H$.
\end{defn}
Note that there is always the regular theory $\T_H$ of a model $H\in \Tmod (\cE)$ obtained adding all regular axioms which are valid in the model: from Lemma~\ref{fadd} we have that $\cE[\T_H]$ is a subtopos of $\cE[\T]$.

\subsection{Nori's construction via construcible $\T$-motives}
Applying the theory of the model to the singular homology $H^{\rm sing}$ (resp. cohomology $H_{\rm sing}$ ) as in \S \ref{singhom} we obtain Nori's effective homological (resp. cohomological) motives as constructible $\T_{H^{\rm sing}}$-motives (resp. $\T^{op}_{H_{\rm sing}}$-motives). See \cite{ABV} and \cite{HMN} for an account on Nori's original construction: we here reformulate it using categorical logic according to \cite{BVCL}.

For a given graph $D$ we have a signature $\Sigma_D$ which attach sorts to objects, function symbols to arrows and for each object of $D$ we also attach sorts and function symbols formalizing an $R$-module structure (with $R$ any ring) as indicated in \cite[\S 2.2]{BVCL}.   For a representation $T:D\to \Rmod$ we can define a regular theory $\T_T$ of $T$ by the set of regular sequents which are valid in $T$. This theory $\T_T$ yields a syntactic category $\cC_{\T_T}^{\rm reg}$.  We also clearly get $\tilde T:D\to \cC_{\T_T}^{\rm reg}$ and since $T$ is a conservative model of $\T_{T}$ we get $F_T:\cC_{\T_T}^{\rm reg}\to \Rmod$ which is exact and faithful. 
\begin{thm}[\cite{BVCL}]\label{BVCL}
The (Barr) exact completion $\cC(T)$ of $\cC_{\T_T}^{\rm reg}$  is an $R$-linear abelian category along with a forgetful (faithful, exact) functor $F_T: \cC(T) \to  \Rmod$ and a representation $\tilde T :D \to  \cC(T)$ such that $F_T\circ\tilde T = T$ universally, \ie the triple $(\tilde T, \cC(T), F_T)$ is initial among such factorizations of the representation $T$.
\end{thm}

For the category of schemes $\Sch_k$ as in \S \ref{singhom} we can take Nori's graph $D^{Nori}$ and Nori's representation $T$ of singular homology for $R=\Z$ (see \cite{ABV} and \cite{HMN}). In this case $$\cC(T)\df {\sf EHM}$$ is Nori's category of effective homological motives. Moreover, the corresponding signature $\Sigma_{D^{Nori}}$ is exactly our signature $\Sigma$ in \S \ref{signature} for $\cC=\Sch_k$ and $\cM$ closed subschemes.
 Dually, for singular cohomology we get Nori's category ${\sf ECM}$ of effective cohomological motives. Therefore, the theory $\T_T$ is exactly the theory of the model $\T_{H^{\rm sing}}$ (resp. $\T_{H_{\rm sing}}^{op}$ for cohomology). Thus:
\begin{cor} \label{Nori}
For the singular homology $H^{\rm sing}$ and cohomology $H_{\rm sing}$ on the category of schemes $\Sch_k$ where $k\into\C$ we have 
$$\cA[\T_{H^{\rm sing}}]\cong {\sf EHM}\ \ \ \text{and}\ \ \ \cA[\T_{H_{\rm sing}}^{op}]\cong {\sf ECM}$$
\end{cor}
We may call $\T_{H^{\rm sing}}$ and $\T_{H_{\rm sing}}^{op}$ the regular singular (co)homology theories. The universal representation $\tilde T: D^{Nori}\to {\sf EHM}$ corresponds to the universal model $H^{\T_{H^{\rm sing}}}\in \Tmod (\cA[\T_{H^{\rm sing}}])$ as in Definition \ref{univhom}. Denote 
\begin{equation}\label{Norihom}
H^{\rm Nori}\in \Tmod ({\sf EHM})
\end{equation}
the model corresponding to the universal model $H^{\T_{H^{\rm sing}}}$ under the equivalence of Lemma \ref{Nori}. 

\begin{remark} A word on the proof of the Theorem \ref{BVCL}. If we are given $D\by{S}\cA\by{F}\Rmod$ and $F\circ S = T$ with $\cA$ abelian and $F$ forgetful then $$S\in \Ttmod(\cA)\cong \text{Lex}_{J_{\T_T}^{\rm reg}}(\cC_{\T_T}^{\rm reg}, \cA) \cong \text{Ex}(\cC (T), \cA)$$
To see that $S$ is a $\T_T$-model in $\cA$ we have used that $F$ is exact and faithful so that it reflects the validity of regular sequents. Note that for a representation $T:D\to \Rmod^{fg}$ with $R$ Noetherian then $F_T: \cC (T) \to  \Rmod^{fg}$ as well.
\end{remark}

\subsection{$\T$-motives}
Recall (see \cite[\S 8.6]{KS}) that for an (essentially small) abelian category $\cA$ we  have that $\Ind (\cA)$ is Grothendieck and the Yoneda embedding yields
$$\cA\into \Ind(\cA) \cong {\sf Lex} (\cA^{op}, \Ab)\subset \A (\cA^{op}, \Ab)\subset \hat{\cA}$$
Note that epi = regular epi = descent = effective descent morphism in an abelian category $\cA$ (actually: regular epis are effective descent in any exact category, see \cite[B1.5.6]{El}). Denote $\Shv (\cA)$
the topos of sheaves for the descent topology. We thus have the following (\cf \cite[Ex. 8.18]{KS}) 
$$\xymatrix{{\sf Lex} (\cA^{op}, \Ab)\ar[r]^-{}  \ar[d]_-{} & \A (\cA^{op}, \Ab) \ar[d]^-{} \\
\Shv (\cA) \ar@{>}[r]^-{f} & \hat{\cA}}$$
2-pull-back diagram of categories where $f$ is the canonical embedding so that: 
\begin{lemma} \label{ind}
$\Ind(\cA[\T]) \cong \Ab (\cE [\T])\cap \A (\cA[\T]^{op}, \Ab) $
\end{lemma}
\begin{proof} It follows from $\cE [\T]\cong \Shv (\cA[\T])$ (see \cite[D3.3.10]{El}).\end{proof}
\begin{defn}
Call $\Ind(\cA[\T])$ the category of \emph{effective $\T$-motives}. 
\end{defn}
\begin{propose}\label{Grothabelian}
The category $\Ind(\cA[\T])$ is a Grothendieck abelian category and
$$ \Tmod (\cA)\cong \text{\sf Ex} (\Ind(\cA[\T]), \cA)$$
for any $\cA$ Grothendieck abelian category.
\end{propose}
\begin{proof} The category $\Ind(\cA[\T])$ is Grothendieck since $\cA[\T]$ is essentially small (see \cite[Thm. 8.6.5 (i) \& (vi)]{KS}). From Proposition \ref{abeliancon} if $H\in\Tmod (\cA)$ yields $r_H:\cA[\T]\to \cA$ exact and $\Ind (r_H): \Ind (\cA[\T])\to \Ind (\cA)=\cA$ is also exact (see \cite[Cor. 8.6.8]{KS}). Conversely, note that $\cA[\T]\into \Ind(\cA[\T])$ is an exact embedding (see \cite[Thm. 8.6.5 (ii)]{KS}) and we are then granted by Proposition \ref{abeliancon}.
\end{proof}
\begin{propose} \label{indquot}
For $H\in \Tmod (\cE)$ then the realization   $$\Ind (\cA[\T])\onto\Ind(\cA[\T_H])$$ is a Serre quotient with a section.
\end{propose}
\begin{proof} As the universal model $H^{\T_H}$ of $\T_H$ in $\cA[\T_H]$ is a $\T$-model then there is an exact functor $\Ind (\cA[\T])\to\Ind(\cA[\T_H])$.  Actually, by Theorem \ref{cltop} this is the restriction of $f^*:\Ab (\cE[\T])\to \Ab (\cE[\T_H])$ as in Lemma \ref{fadd}. Using Lemma \ref{ind} we get that $f^*f_*\cong id$ for the Ind-categories as well. 
\end{proof}

\subsection{$\T$-motivic complexes}\label{complexes}
The following constructions are intended for $\cC$ the category of schemes $\Sch_k$ and $\cM$ the subcategory given by closed subschemes. However, the following applies to CW-complexes as well so that we keep some arguments in the categorical setting for the sake of the interested reader keeping in mind the parallel assumptions of \S\ref{singhom} and \S\ref{algsinghom} as a main reference.

Assume that $\cC$ has an initial object $\emptyset$ keeping the notation (and assumptions) as in \eqref{zero}. We have  $\emptyset \to X$ in $\cM$ for each $X$ object of $\cC$ and if $X\to \emptyset$ is in $\cM$ then $X\cong\emptyset$.  Thus $\cM \subset \cC$  has the same objects of $\cC$ and we will assume that we can suitably filter an object $X$ of $\cC$ by maps in $\cM$. 

Further assume that $\cM\subset \cC$ is a subcategory of distinguished monos, \ie we have  $$\text{Iso}(\cC)\subset {\rm Morph} (\cM) \subseteq \text{Mono}(\cC)$$ and also  $n, m\in \cM$ and $m =na$ implies $a\in \cM$.  Let $Y\subseteq X$ be the suboject determined by a mono $Y\into X$ in $\cM$. Denote $\Sub_{\cM}(X)$ the poset of $\cM$-subobjects of $X$. 

Assume that $\cM$ is stable by direct images: given $f: X\to X'$ morphism of $\cC$ there is a smallest $\cM$-factorization, \ie we have
$$\xymatrix{X \ar[r] \ar@/^1.7pc/[rr]^-{f} & \im f \ar[r]^-{m'} & X'}$$
with $m'\in \cM$ and minimal among such factorizations. For $Y\subseteq X$ let $f_*(Y)$ be the suboject determined by $ \im (fm)$ where $m: Y\into X$ represents $Y\subseteq X$.  We then have that $\Sub_{\cM}$ with $f\leadsto f_*$ is a covariant functor (see \cite{LBVD}). Suppose that we have joins $Y\cup Y'$ of $\cM$-subobjects of $X$ so that $\Sub_{\cM}(X)$ is a join-semilattice and a directed poset. For example, this is the case of schemes where $\cM$-subobjects are closed subschemes and $f_*(Y)$ is given by the closure of the image $f(Y)$.

\begin{remark}However, note that we can have $f$ surjective, \ie $\im f =X'$, but $f$ not epi and $f_*:\Sub_{\cM}(X)\to \Sub_{\cM}(X')$ not surjective as a mapping. Assume that $\cM$ is also stable by inverse images, \ie $Y'\subseteq X'$ we have $$f^*(Y')\df X\times _{X'}Y'\subseteq X$$ so that $\Sub_{\cM}$  with $f\leadsto f^*$ is a contravariant functor. Then $f_*f^* = id$ for $f$ surjective if and only if $f$ surjective implies $f_*$ surjective (see \cite{LBVD}).  
\end{remark}

Let $H\in \Tmod (\cA)$ for $\cA$ an (essentially small) abelian category. Note that by the proof of Lemma \ref{purity} it follows that $H_*(X, Y)$ depends of the $\cM$-subobject $Y\subseteq X$ only. Taking the filtered inductive limit on $Y\subsetneq X$ we get
$$H_*(X) \to \limdir{Y} H_*(X, Y)$$
a morphism of $\Ind(\cA)$. For $f : X\to X'$ we get
$$ \limdir{Y} H_*(X, Y)\to \limdir{Y'} H_*(X', Y')$$
induced by $f_*:\Sub_{\cM}(X)\to \Sub_{\cM}(X')$ and we get a family of functors $$\limdir{ }H_*:\cC\to \Ind (\cA)$$

Assume given a suitable ``dimension'' function on $\cC$ or just let $\cC$ be the category $\Sch_k$ or the subcategory ${\rm Aff}_k$ of affine schemes. Suppose that each object $X$ of $\cC$ is provided with finite exhaustive filtrations $X_{\d}$ of ``dimensional type'' 
 $$X_{d+1}=X= X_d\supset \cdots X_p \supset X_{p-1}\cdots \supset X_0 \supset X_{-1}=\emptyset =X_{-2}$$
where $X_p$ has ``dimension'' at most $p$ and $X_p\in \Sub_{\cM}(X)$. Suppose that the inductive system of all such filtrations is filtered and functorial, \eg it is stable under $\cup$ and direct images. Then for $q\in \Z$ fixed we get
$$\partial_{p+q}: \limdir{X_{p-1}\subset X_p} H_{p+q}(X_{p}, X_{p-1})\to  \limdir{X_{p-2}\subset X_{p-1}} H_{p+q-1}(X_{p-1}, X_{p-2})$$
defining a complex of $\Ind (\cA)$ by Lemma \ref{2=0} depending functorially on $X$. Moreover, we get the Grothendieck ``niveau'' spectral sequence.

\begin{lemma}\label{niveau} In the Grothendieck category $\Ind (\cA)$ there is a convergent homological spectral sequence 
$$E^1_{p , q}(X)\df \limdir{X_{p-1}\subset X_p} H_{p+q}(X_{p}, X_{p-1})\Rightarrow H_{p+q}(X)$$
with induced ``niveau'' filtration $$N_pH_{n}(X)\df \im (\limdir{X_p} H_{n}(X_p)\to H_{n}(X))$$
\end{lemma}
\begin{proof} This spectral sequence can be obtained by making use of a standard exact couple as in \cite[\S 3]{BO} and its convergence is granted by Lemma \ref{purity}.
\end{proof} 

Consider the double complex $E^1_{* , *}(X)$ with zero vertical differentials and $$C_{\d}^H(X)\df {\rm Tot}\ E^1_{* , *}(X)$$ the associated total complex together with an augmentation
$$``\bigoplus_{n\in \Z}{} " H_{n}(X)[n] \to C_{\d}^H(X)$$ Moreover $$X\in \cC\leadsto C_{\d}^H(X)\in \Ch(\Ind (\cA))$$ is functorial as for $f :X\to X'$ we have  
$E^1_{* , *}(X)\to  E^1_{* , *}(X')$ from the naturality of $\partial_*$. 

\begin{propose}\label{eqrel} Let $H\in \Tmod (\cA)$ for $\cA$ an (essentially small) abelian category. The there is an exact realization  functor
$$\Ch(\Ind (\cA[\T]))\to \Ch(\Ind (\cA))$$
which is sending $C_{\d}^{H^\T}(X)\leadsto C_{\d}^H(X)$.
\end{propose}
\begin{proof}
Note that for $\cA[\T]$ and $H^\T\in \Tmod (\cA[\T])$ the universal homology in Definition \ref{univhom} we get an exact realization functor $r_H:\cA [\T] \to \cA$ induced by $H$ and sending $H^\T\leadsto H$. The induced exact functor $\Ind(\cA [\T]) \to \Ind(\cA)$ yields the claimed functor.
\end{proof}
 
\begin{lemma} \label{cellular}
Assume $H\in \Tmod (\cA)$ and $\cC$ such that $$\limdir{X_{p-1}\subset X_p} H_{p+q}(X_{p}, X_{p-1})=0$$ 
for $q\neq 0$. Then the complex $C_\d^H (X)$ is given by the following bounded complex 
$$ \cdots \to \limdir{X_{p-1}\subset X_p} H_{p}(X_{p}, X_{p-1})\to  \limdir{X_{p-2}\subset X_{p-1}} H_{p-1}(X_{p-1}, X_{p-2})\to \cdots $$
concentrated between $0$ and $d = \dim (X)$  and $$H_n(C_{\d}^H(X))\cong H_n(X)\in \cA$$
Moreover $C_{\d}^H(X)\in  D^b(\cA)$ and $H_n(X)\neq 0$ implies  $0\leq n\leq d$.
\end{lemma}
\begin{proof} It follows from Lemma \ref{niveau}. The spectral sequence $E^r_{p , q}(X)$ degenerates at $E^2_{p , q}$ as $E^1_{p , q}(X) =0$ for $q\neq 0$. Thus $C_{p}^H(X)= E^1_{p , 0}(X)$ and  $E^2_{p , 0}(X)= H_p(C_{\d}^H(X))\cong H_n(X)$ for $n=p$. Finally, let  $D^b_{\cA}(\Ind(\cA))$ be the triangulated subcategory of $D^b(\Ind(\cA))$ determined by those bounded complexes whose homology is in $\cA$. We have that 
$D^b(\cA) \by{\sim}D^b_{\cA}(\Ind(\cA))$ and
thus $C_{\d}^H(X)\in D^b(\cA)$ (see \cite[Thm. 15.3.1 (i)]{KS}). 
\end{proof}
Consider the case of Nori's homology $H^{\rm Nori}\in \Tmod ({\sf EHM})$ in \eqref{Norihom}. Lemma \ref{cellular} holds for $H^{\rm Nori}$ and ${\rm Aff}_k$ affine schemes: by the ``basic Lemma" affine schemes and singular homology are provided with a cofinal system of filtrations given by ``good pairs'' (see \cite[\S 2.5 \& \S 8.2]{HMN}). 

Let ${\rm Aff}_X$ be the category of affine schemes over $X$ in $\Sch_k$ and let $\phi_X:{\rm Aff}_X \to {\rm Aff}_k$ be the forgetful functor. Nori's motivic functor is
$$X\in \Sch_k \leadsto M (X)\df {\rm Tot}\ {\rm Nerve}\ (C_{\d}^{H^{\rm Nori}}\circ\phi_X)\in \Ch (\Ind ({\sf EHM})) $$
given by the composition of Nerve and Tot functors. Note that for $X =\Spec (A)$ affine we have 
$M (X) \cong C_{\d}^{H^{\rm Nori}}(X)$.
\begin{propose} \label{Nmotf}
There is an exact realization functor 
$$r^{\rm Nori}: \Ch(\Ind (\cA[\T]))\to \Ch(\Ind ({\sf EHM}))$$
and a factorization of Nori's motivic functor
$$\xymatrix{\Sch_k\ar[r]^-{C} \ar@/_/[dr]_-{M} & \Ch(\Ind (\cA[\T]))\ar[d]^-{r^{\rm Nori}}\\ 
& \Ch(\Ind ({\sf EHM}))} $$
Furthermore, the functor $r^{\rm Nori}$ is a Serre quotient with a section.
\end{propose}
\begin{proof} The realization $r^{\rm Nori}$ is given by Proposition \ref{eqrel} with $\cA = {\sf EHM}$ and $H^{\rm Nori}\in \Tmod ({\sf EHM})$.  Following the original argument due to Nori define
$$C (X)\df {\rm Tot}\ {\rm Nerve}\ (C_{\d}^{H^\T}\circ\phi_X)$$
and get the functor $C$ in the claimed factorisation. In fact, since $$\xymatrix{{\rm Aff}_k\ar[r]^-{C.^{H^\T}} \ar@/_/[dr]_-{C.^{H^{\rm Nori}}} & \Ch(\Ind (\cA[\T]))\ar[d]^-{r^{\rm Nori}}\\ 
& \Ch(\Ind ({\sf EHM}))} $$
commutes therefore the claimed commutativity holds as well. The last claim follows from Proposition \ref{indquot}.
\end{proof}

Consider, similarly, the case of Suslin-Voevodsky singular homology $H^{\rm SV}\in \Tmod ({\rm PST})$ as in \S \ref{algsinghom}. We get a realization by Proposition \ref{eqrel} $$r^{\rm SV}: \Ch(\Ind (\cA[\T]))\to \Ch({\rm PST})$$ (\cf Proposition \ref{Grothabelian}). This functor further localize: 
\begin{propose}\label{DM}
There is a realization triangulated functor
$$D(\Ind (\cA[\T]))\to \DM^\eff$$
where $\DM^\eff$ is the (unbounded) triangulated category of Voevodsky effective motivic complexes. 
\end{propose} 
\begin{proof} Note that by sheafification (see \cite[Thm. 13.1]{VL}) we can prolong $r^{\rm SV}$ with target $\Ch({\rm NST})$ for Nisnevich sheaves with transfers NST and we get $D(\Ind (\cA[\T]))\to D({\rm NST})$.  Recall that $\DM^\eff\subset D({\rm NST})$ is a Bousfield localization of $D({\rm NST})$ and the claimed realization is then obtained by composition with the $\Aff^1$-localization functor which is left adjoint to the inclusion (see \cite[Thm. 14.1]{VL}).\end{proof}
\begin{remark} A new $t$-structure on the $\Q$-linearized category $\DM^\eff_\Q$ shall be obtained by showing that the category $\DM^\eff_\Q$ is a Bousfield localization of the model category $\Ch(\Ind (\cA[\T]))_\Q$. Actually, these categories share a common ``motivic'' $t$-structure as we can see from the following chain of equivalences (see \cite{Ay}, \cite{ABV} and \cite{BVK})
$$\DM^\eff_{\leq 1}\cong D(\Ind (\cM_1^\Q))\cong  D(\Ind ({\sf EHM_1^\Q}))$$
where $\DM^\eff_{\leq 1}$ is the smallest subcategory
of $\DM^\eff_\Q$ closed under infinite sums generated by the motives of curves, $\cM_1^\Q$ is the abelian category of Deligne $1$-motives up to isogenies and ${\sf EHM_1^\Q}$ is the abelian subcategory of ${\sf EHM^\Q}$ generated by
the $i$-th Nori's homologies $H_i^{\rm Nori}(X,Y)$ for $i \in \{0, 1\}$.
\end{remark}
\subsubsection*{Acknowledgements} 
I am very grateful to Olivia Caramello, Silvio Ghilardi and Laurent Lafforgue for many helpful discussions.


\end{document}